\documentclass[10pt]{amsart}
\usepackage{amsmath, amsthm, amssymb, amsfonts, verbatim}

\usepackage{array}
\usepackage[bookmarks,colorlinks,linkcolor=BrickRed,citecolor=OliveGreen]{hyperref}
\usepackage{amscd}
\usepackage{epsfig}
\usepackage[usenames,dvipsnames]{xcolor}
\usepackage{setspace}
\usepackage{multirow}
\usepackage{mathrsfs}
\usepackage{url}
\usepackage{todonotes}

\numberwithin{equation}{section}

\newcommand{\ra}{\rightarrow}
\newcommand{\half}{\frac{1}{2}}
\newcommand{\e}{\epsilon}

\newcommand{\R}{{\mathbb R}}

\newcommand{\V}{{\mathbb V}}
\newcommand{\K}{{\mathbb K}}

\newcommand{\M}{{\mathbb M}}
\newcommand{\X}{{\mathbb X}}

\newcommand{\grad}{\operatorname{grad}}

\renewcommand{\div}{\operatorname{div}}

\newcommand{\skw}{\operatorname{skw}}

\newcommand{\tr}{\operatorname{tr}}

\newcommand{\lap}{\Delta}
\newcommand{\pd}{\partial}

\newcommand{\es}{e_\sigma}

\newtheorem{thm}{Theorem}[section]
 
\newtheorem{lemma}[thm]{Lemma}

\theoremstyle{definition}

\newtheorem{eg}[thm]{Example}

\theoremstyle{remark}

\begin{document}

\title{Mixed methods for elastodynamics with weak symmetry}
\author{Douglas N. Arnold\\
Jeonghun J. Lee}
 \thanks{The first author was supported by NSF grant
DMS-1115291.  The authors are grateful for computational resources from the University
of Minnestoa Supercomputing Institute used in this research.} 
\begin{abstract}
We analyze the application to
elastodynamic problems of mixed finite element methods for elasticity with weak symmetry.  Our approach
leads to a semidiscrete method which consists of a system of ordinary differential equations without
algebraic constraints.  Our error analysis, which is based on a new elliptic projection operator,
applies to several mixed finite element spaces developed for elastostatics. The error estimates
we obtain are robust for nearly incompressible materials.
\end{abstract}



\maketitle

\section{Introduction} 
The linear elastodynamics equation describes wave propagation in an elastic medium. It has the form  
\begin{align}
\label{ed-eq1} 
\rho \ddot u - \div C \epsilon(u) =f \quad \text{in $\Omega$},
\end{align}
where $u : \Omega \ra \R^n$ is the unknown displacement vector field, $\epsilon(u)$ the corresponding linearized strain tensor, $C$ the stiffness tensor of the elastic medium, $\rho$ the mass density, and $f$ an external body force. In \eqref{ed-eq1} we have suppressed the dependence on the independent variables for simplicity, but all the quantities appearing in the
equation may depend on $x\in\Omega$, and $t\in[0,T_0]$ (for some positive $T_0$), and the equation is
supposed to hold for all such $t$. Combining equation \eqref{ed-eq1}
with initial conditions $u=u_0$, $\dot u=v_0$ at time $t=0$ and with appropriate boundary conditions,
we obtain a well-posed problem (see e.g., \cite{Duvaut-Lions-book}, Theorem 4.1, or section~\ref{section:weak-form} below). 

Mixed finite element methods, in which the stress $\sigma=C\epsilon(u)$ and displacement $u$ are approximated independently, are
popular for the numerical approximation of elastostatic problems.  The application of mixed methods to elastodynamic problems
has been studied by various researchers as well. In \cite{DougGup86}, Douglas and Gupta studied plane
linear elastodynamics using the mixed finite elements developed in \cite{ADG84}. In \cite{Mak92}, Makridakis analyzed
mixed finite elements for elastodynamics in both two and three dimensions, including
higher order time discretization, using the elements of \cite{ADG84, Johnson_Mercier, Sten88}. In \cite{BJT02}, B\'{e}cache, Joly, and Tsogka developed a new family of rectangular mixed finite elements and studied the a priori error analysis.

These mixed finite element spaces incorporate the symmetry of the stress tensor into the finite element
space, and as a result are rather complicated.  This has led to a great deal of interest in mixed finite element for elasticity in which the symmetry of the stress is imposed only weakly.  This idea was first suggested in \cite{Fraeijs}
and elements based on it were first developed in \cite{AmaraThomas} and \cite{ABD84}.
Recently a great deal of progress has been made in stable mixed finite elements for elasticity with weak symmetry.
In this paper, we study the application of such elements to linear elastodynamics.  In particular,
we treat in a unified fashion the elements of Arnold, Falk, and Winther \cite{AFW07} and the variant introduced by
Cockburn, Gopalakrishnan, and Guzm\'an \cite{CGG10}, as well as another method of Gopalakrishnan and Guzm\'an
\cite{GG10} and a related older method of Stenberg \cite{Sten88}.  Although we only consider the case of
elastodynamics, we point out that one advantage of mixed finite elements is that they can be easily
extended to materials with more complex constitutive equations, such as viscoelasticity (see \cite{LeeThesis}
and \cite{RogWin10}, where the quasistatic problem is considered), and likely also to plasticity and poroelasticity.

Since symmetry of the stress tensor is an algebraic condition, the most obvious formulation
of elastodynamics with weak symmetry leads, after spatial discretization, to a system of
differential--algebraic equations.  Indeed, that is the approach taken in \cite{RogWin10} for quasistatic
viscoelasticity.  However, in this paper we propose a different mixed variational formulation
for elastodynamics with weak symmetry (see \eqref{ed-weaksym1}--\eqref{ed-weaksym3}).  Our approach leads simply to a system of 
ordinary differential equations in time after spatial discretization.  Therefore, standard time stepping methods
can be applied, and the analysis of the temporal discretization is standard.  For that reason we focus on
the spatial discretization in this paper.

The remainder of the paper is organized as follows. In section~\ref{section:notations}, we set out notations and
 describe the features of mixed finite elements for elasticity with weak symmetry of stress which we will need
for analysis of elastodynamic problems. In section~\ref{section:weak-form}, we prove well-posedness of linear elastodynamics using the Hille--Yosida theorem and derive the weak formulation of it which we will use
for discretization. In section~\ref{section:semidiscrete}, we analyze the semidiscretization,
and obtain a priori error estimates for the elements of \cite{AFW07} and \cite{CGG10}. In this context,
we also prove that numerical solution is free from locking
in the nearly incompressible regime, i.e., the constants in the error bounds do not grow unboundedly as the Lam\'e coefficient $\lambda$ tends to infinity. In section~\ref{section:imp-err-analysis}, we give an improved error analysis for the elements of \cite{GG10} and \cite{Sten88}. Finally, we present numerical results supporting the analysis in the final section.

\section{Notations and preliminaries} \label{section:notations}
\subsection{Notations}
Let $\Omega$ be a bounded smooth domain in $\R^n$ with $n=2$ or $3$.
We use $\V$ to denote the space $\R^n$ of $n$-vectors and $\M$, $\mathbb{S}$, and $\K$ to denote
the space of all, symmetric, and skew-symmetric $n\times n$ matrices, respectively.
The components of a vector field $u : \Omega \ra \mathbb{V}$ and a matrix field $\sigma : \Omega \ra \mathbb{M}$
are denoted by $u_i$ and $\sigma_{ij}$, respectively. The $L^2$ inner products on vector and matrix fields
are given by
\begin{gather*}
( v, w ) = \int_{\Omega} v \cdot w \,dx = \int_{\Omega} \sum_{i=1}^n v_i w_i \,dx , \quad v, w : \Omega \rightarrow \mathbb{V},
\\
( \sigma , \tau) = \int_{\Omega} \sigma : \tau \,dx= \int_{\Omega} \sum_{1 \leq i,j \leq n} \sigma_{ij} \tau_{ij} \,dx,
\quad\sigma, \tau : \Omega \rightarrow \mathbb{M}.
\end{gather*}
We denote the corresponding norms by $\| \sigma \|$, $\| u \|$ and the corresponding Hilbert spaces by
$L^2(\Omega; \mathbb{M})$, $L^2(\Omega; \mathbb{V})$.
For $\sigma : \Omega \rightarrow \mathbb{M}$ and $u : \Omega \rightarrow \mathbb{V}$, $\div \sigma$ and $\grad u$ are defined as the row-wise divergence and the row-wise gradient
\begin{align*}
(\div \sigma)_i = \sum_j \pd_j \sigma_{ij}, \qquad (\grad u)_{ij} = \pd_j u_i, 
\end{align*}
respectively, where $\pd_j$ denotes the $j$th partial derivative, applied in the sense of distributions. For $\sigma : \Omega \ra \M$, the skew-symmetric part of $\sigma$ is
$\skw \sigma = (\sigma -\sigma^{T})/2$.

We use standard notation for the Sobolev space $H^m (\Omega)$, $m \geq 0$,  with norm $\| \cdot \|_m$, and the space $\mathring{H}^1(\Omega)$ of $H^1(\Omega)$ functions with vanishing trace. For $\X=\V, \M, \K$, or $\mathbb{S}$,
we write $H^m(\Omega; \X)$ for the space of $\X$-valued fields such that each component belongs to $H^m(\Omega)$. If $\X$ is clear in context, we may write $H^m (\Omega)$ instead $H^m (\Omega; \X)$.  
For $\X$ a subspace of $\M$,  let
\begin{align*}
H(\div, \Omega; \X) = \{ \sigma \in L^2 (\Omega; \X ) \; | \; \div\sigma\in L^2(\Omega;\V)\; \}, 
\end{align*}
which is a Hilbert space with the norm $\| \sigma \|_{\div}^2 = \| \sigma \|^2 + \| \div \sigma \|^2$.  We abbreviate
\begin{gather} \label{msvk}
\begin{split}
M =  H(\div, \Omega; \mathbb{M}), \quad S =  H(\div, \Omega; \mathbb{S}), \quad V = L^2(\Omega; \mathbb{V}), \quad K = L^2(\Omega; \mathbb{K}).
\end{split}
\end{gather}

Let $\mathcal{X}$ be a Banach space (e.g., one of the Hilbert spaces defined above), $T_0$ a positive real number,
$m$ a nonnegative integer, and $1\le p\le \infty$.
We denote by $L^p([0,T_0];\mathcal X)$ or $L^p\mathcal X$ the space of functions $f:[0,T_0]\to \mathcal X$ for which
\begin{align*}
\| f \|_{L^p\mathcal{X}}^p := 
\int_0^{T_0}\| f \|_\mathcal{X}^p \,dt  <\infty,
\end{align*}
(with the usual modification for $p=\infty$), and by $W^{m,p}([0,T_0];\mathcal X)$ or $W^{m,p}\mathcal X$
the space for which
$$
\| u \|_{W^{m,p} \mathcal{X}}^p := \sum_{l=0}^m \| \pd^l u / \pd t^l \|_{L^p \mathcal{X}}^p<\infty.
$$
We shall also use the space $C^m([0,T_0];\mathcal X)$ of $m$-times continuously differentiable functions.

For brevity of notation,
we write $\| f, g \|_\mathcal{X}$ to denote $\| f \|_\mathcal{X} + \| g \|_\mathcal{X}$ when $f$ and $g$
both belong to some Banach space $\mathcal{X}$, and, as we have seen, we use $\dot{f}$, $\ddot f$, \dots, 
to denote $\pd f / \pd t$, ${\pd^2 f}/ { \pd t^2}$, etc.

\subsection{Mixed formulations of linear elastostatics}  \label{section:stationary}
In this section we review the discretization of \emph{stationary} linear elasticity using mixed finite elements with weak symmetry.  For details, see \cite{AFW07}.
The constitutive equation of linear elasticity is $\sigma=C\epsilon(u)$, where, for a given
displacement vector field $u$, the linearized strain tensor $\epsilon(u)$ is given by 
$\e(u) = (\grad u + (\grad u)^T)/2$,
and at each point $x$ the  elasticity tensor
$C(x)$ is a symmetric positive definite linear operator from $\mathbb{S}$ to $\mathbb{S}$,  bounded above and below.
The same then holds for the compliance tensor $A:=C^{-1}$. For a homogeneous isotropic
elastic material
\begin{align}
\label{compliance} C \tau = 2\mu \tau + \lambda \tr(\tau)I,\quad
A\tau = \frac{1}{2\mu}\left(\tau-\frac{\lambda}{2\mu+ n\lambda}\tr(\tau)I\right),
\end{align}
where $\mu$, $\lambda$ are positive scalars called the Lam\'{e} coefficients, and $\tr(\tau)$ is the trace of $\tau$. 

The boundary value problem of linear elastostatics consists of the constitutive equation,
the equilibrium equation $-\div\sigma=f$, where $f$ is a given body force density, and boundary conditions.
For simplicity, we only consider problems with the homogeneous displacement boundary conditions,
although it is not difficult to extend our approach to more general boundary conditions. Thus the elastostatic problem is
\begin{gather*}
A\sigma = \e(u), \quad -\div \sigma = f  \quad \text{in } \Omega, \qquad u = 0  \quad \text{on } \pd \Omega.
\end{gather*}
Integrating by parts, we obtain a weak formulation of linear elasticity with strongly imposed symmetry.
It seeks $(\sigma, u)$ in $S \times V$ so that
\begin{align}
\label{es-HR1} (A \sigma, \tau) + (\div \tau, u) &= 0, & & \tau \in S, \\
\label{es-HR2} - (\div \sigma, w) &= (f, w) , & & w \in V.
\end{align}
For any $f\in L^2(\Omega;\V)$ this system admits a unique solution.

We now modify this formulation to impose the symmetry of stress weakly.
For this, we extend the operator $A$, originally defined only on symmetric tensors, to map $\M\to\M$,
by setting it equal to the identity map on skew-symmetric tensors (or a positive multiple
of the identity map). Next, we introduce the
rotation field, $r := \skw \grad u$. 
Then the triple $(\sigma, u, r)$ in $M \times V \times K$ satisfies
\begin{align*}
(A \sigma, \tau) = (\grad u - r, \tau) = -(u, \div \tau) - (r, \tau), \quad \tau \in M.
\end{align*}
Now we seek $(\sigma, u, r)$ in $M \times V \times K$ satisfying
\begin{align}
\label{es-weaksym-HR1}  (A \sigma, \tau) + (\div \tau, u) + (r, \tau) &= 0 , & & \tau \in M, \\
\label{es-weaksym-HR2} - (\div \sigma, w)  &= ( f, w ), & & w \in V , \\
\label{es-weaksym-HR3} (\sigma, q) &= 0, & & q \in K,
\end{align}
where the third equation expresses the symmetry of the stress.
The formulation (\ref{es-weaksym-HR1}--\ref{es-weaksym-HR3}) admits a unique solution,
for which $(\sigma,u)$ coincides with the solution of (\ref{es-HR1}--\ref{es-HR2}) and $r=\skw\grad u$.
The triple $(\sigma,u,r)$ may be characterized variationally as the unique critical point of the
functional
$$
(\sigma,u,r)\mapsto \frac12 (A\sigma,\sigma) + (\div\sigma,u)+(\sigma,r) +(f,u)
$$
over $M\times V\times K$.

\subsection{Mixed finite elements for elastostatics with weak symmetry}\label{section:mfem}
A mixed method for elastostatics with weak symmetry is a Galerkin method based on this weak formulation.
Thus we make a choice of finite element subspaces $M_h \subset M$, $V_h \subset V$, $K_h \subset K$
and seek $(\sigma_h, u_h, r_h)$ in $M_h \times V_h \times K_h$ so that
\begin{align}
\label{es-weaksym-HR1-discrete}  (A \sigma_h, \tau) + (\div \tau, u_h) + (r_h, \tau) &= 0 , & & \tau \in M_h, \\
\label{es-weaksym-HR2-discrete} - (\div \sigma_h, w)  &= ( f, w ), & & w \in V_h , \\
\label{es-weaksym-HR3-discrete} (\sigma_h, q) &= 0, & & q \in K_h.
\end{align}
Of course the spaces must be suitably chosen to insure that this finite dimensional problem is
nonsingular, and to obtain error estimates.  In this paper, we shall consider four families
of such spaces, each based on a simplicial triangulation of $\Omega$ into elements,
and a choice of polynomial degree $k>0$.  The simplest choice is the AFW family defined in
\cite{AFW07}.  For any $k\ge 1$, the spaces $V_h$ and $K_h$ are simply defined as the
fields in $V$ and $K$ which are piecewise polynomial of degree at most $k-1$ on each element.  That
is, the shape function space for each component is $\mathcal P_{k-1}$, with
no interelement continuity imposed.  The stress space $M_h$ consists of all matrix fields
in $M$ which belong piecewise to $\mathcal P_k$.  For these elements, all three
variables, $\sigma$, $u$, and $r$, are approximated with an error $O(h^k)$ in $L^2$.  This
is clearly the best permitted by the subspaces for $u$ and $r$, but the inclusion of $\mathcal P_k$
in the shape functions for $\sigma$ suggests the possibility of $O(h^{k+1})$ for $\sigma$.  This however
does not hold.  This observation motivated the CGG elements of \cite{CGG10}, which
take the same spaces for $V_h$ and $K_h$, but replace the shape function space for $M_h$ with a space which
is strictly smaller than $\mathcal{P}_k$ but which still contains $\mathcal{P}_{k-1}$.  These elements were shown to
satisfy the same error estimates as the AFW elements.

The elements in \cite{GG10} go in the other direction, increasing the AFW spaces to obtain a higher
rate of convergence.  The displacement space remains piecewise $\mathcal P_{k-1}$, but the rotation
space is increased to piecewise $\mathcal P_k$, and the stress space on each element consists
of $\mathcal P_k$ plus a number of higher degree bubble functions.  The same approach was taken by
Stenberg \cite{Sten88}, although a larger number of bubble functions were used and he required $k\ge 2$.
The four methods are summarized in Table~\ref{mixed-approx}.
\begin{table}[t]  
\caption{Complete polynomial degree included in the shape function spaces for various mixed methods.} \label{mixed-approx}
\centering
\begin{tabular}{>{\small}c >{\small}c >{\small}c >{\small}c >{\small}c}
\hline
\multirow{2}{*}{elements}	&  \multicolumn{3}{>{\small}c}{approximability} & \multirow{2}{*}{order } \\ 
						&$\sigma$	& $u$	&  $r$		&	\\  \hline \hline
AFW \cite{AFW07}			& $k$		& $k-1$	& $k-1$		& $k \geq 1$ \\ 
CGG \cite{CGG10}			& $k-1$		& $k-1$	& $k-1$		& $k \geq 1$ \\ 
Stenberg \cite{Sten88}		& $k$		& $k-1$	& $k$ 		& $k \geq 2$ \\ 
GG \cite{GG10}			& $k$		& $k-1$	& $k$		& $k \geq 1$ \\ \hline
\end{tabular}                           
\end{table}

These methods share a number of common features which will allow us to analyze them
in a unified fashion.  Each satisfies the stability conditions
\begin{itemize}
 \item[\bf (A0)] $\div M_h = V_h$,
 \item[\bf (A1)]  There exists $c>0$ so that for any $(u, r) \in V_h \times K_h$, there is $\tau \in M_h$ with
\begin{align*}
\div \tau = u, \quad (\tau, q) = (r, q), \quad \forall q \in K_h, \qquad \| \tau \|_{\div} \leq c(\|u\| + \|r \|). 
\end{align*}
\end{itemize}
These conditions imply that the mixed method is stable in the sense of Brezzi, and so admits a unique solution.
In order to get the best estimates, however, more structure is used.
For each of the methods there is a natural interpolation operator $\Pi_h : H^1(\Omega; \M) \ra M_h$ which satisfies
the commutativity condition
\begin{itemize}
 \item[\bf (A2)] $\div \Pi_h \sigma = P_h \div \sigma$, \quad$\sigma \in H^1(\Omega;\M)$,
\end{itemize}
where $P_h$ is the $L^2$ projection onto $V_h$.  We also denote by $P_h'$ the $L^2$ projection onto $K_h$.
The projection operator $\Pi_h$ is defined element by element and preserves the finite element space, so satisfies the error estimates
\begin{equation}\label{piest}
\| \sigma - \Pi_h \sigma \| \leq ch^m \| \sigma \|_m, \quad  1\le m \leq
\begin{cases}
k,  &\text{CGG}, \\
k+1, & \text{AFW, Stenberg, GG}.
\end{cases}
\end{equation}
The conditions {\bf (A0)}, {\bf (A1)}, and {\bf (A2)} imply the following error estimates, which improve on the basic
stability estimates:
\begin{gather}
\label{es-imp-err2} \| \sigma - \sigma_h \| + \| P_h u - u_h \| + \| r - r_h \| \leq c(\| \sigma - \Pi_h \sigma \| + \| r - P_h' r \|).
\end{gather}
These improved estimates appeared in \cite{AFW07, GG10, CGG10}
and an equivalent result for $\| \sigma - \sigma_h \|$ and $\| r - r_h \|$ is obtained in \cite{Sten88}. We refer to \cite{Guzman11} for a unified analysis. Combining this estimate with the approximation rates implied by Table~\ref{mixed-approx}, 
we obtain error bounds
\begin{align*}
\| \sigma - \sigma_h \| + \| P_h u - u_h \| + \| r - r_h \| \leq ch^m (\| \sigma \|_m + \| r \|_m), 
 \qquad 1 \leq m \leq \bar k,
\end{align*}
where $\bar k=k$ for the AFW and CGG elements, and $\bar k=k+1$
for the Stenberg and GG elements.

\subsection{The weakly symmetric elliptic projection operator}
Our error analysis for linear elastodynamics, will depend on a bounded projection $\tilde \Pi_h : M \ra M_h$
which we now define.  Let $M_h$, $V_h$, $K_h$ be one of the choices of finite element spaces
discussed in the previous section, and $\Pi_h:H^1(\Omega;\M)\to M_h$ the corresponding interpolant.
Given $\sigma \in M$, there exists a unique triple
$(\sigma_h, u_h, r_h) \in M_h \times V_h \times K_h$ such that
\begin{align} 
\label{es-weaksym-proj-eq1} ( \sigma_h, \tau) + (\div \tau, u_h) + (\tau, r_h) &= ( \sigma, \tau) ,  \qquad \quad \tau \in M_h,   \\
\label{es-weaksym-proj-eq2} (\div \sigma_h, w)  &= (\div \sigma, w) ,   \quad \;w \in V_h , \\
\label{es-weaksym-proj-eq3} (\sigma_h, q) &= (\sigma, q), \qquad \quad  q \in K_h .
\end{align}
In other words, $(\sigma_h, u_h, r_h)$ is the mixed method approximation of $(\sigma, 0,0)$.
We define $\tilde \Pi_h \sigma=\sigma_h$.  If $\sigma\in M_h$, we clearly have $\sigma_h=\sigma$, $u_h=0$, $r_h=0$,
so $\tilde\Pi_h$ is a projection.
We now establish some additional properties.

\begin{lemma} \label{weak-sym-elliptic}  For one of the mixed methods given in Table \ref{mixed-approx},
let $M_h$ be the stress space, $\Pi_h:H^1(\Omega; \mathbb{M}) \ra M_h$ the corresponding projection satisfying {\bf(A2)},
and $\tilde\Pi_h:M\to M_h$ the elliptic projection just defined.  Then
\begin{equation}\label{es-weaksym-proj1}
\div \tilde \Pi_h \sigma = P_h \div \sigma, \quad (\tilde \Pi_h \sigma, q) = (\sigma , q), \quad \sigma\in M, \ q \in K_h.
\end{equation}
Moreover, there exists a constant $c$ such that
\begin{equation} 
\label{es-weaksym-proj2} \| \tilde \Pi_h \sigma \|_{\div} \leq c \| \sigma \|_{\div}, \quad \| \sigma -  \tilde \Pi_h \sigma \| \leq c \| \sigma - \Pi_h \sigma \|, \quad \sigma\in H^1(\Omega;\M).
\end{equation}
\end{lemma}
\begin{proof} 
The properties in \eqref{es-weaksym-proj1} are immediate from \eqref{es-weaksym-proj-eq2} and \eqref{es-weaksym-proj-eq3}
in the definition of the elliptic projection, and the fact {\bf(A0)} that $\div M_h=V_h$.  The first estimate
in \eqref{es-weaksym-proj2} is a consequence of the Brezzi stability and the second estimate is just the error estimate \eqref{es-imp-err2} in the case $u=0$ and $r=0$.
\end{proof}

\section{Weak formulation of elastodynamics with weak symmetry} \label{section:weak-form}
In this section we derive a velocity-stress formulation of linear elastodynamics with weakly imposed symmetry of stress and show that it is well-posed. For simplicity, we only consider homogeneous displacement boundary conditions. 

In order to have a mixed form with velocity and stress, we set $v = \dot u$, $\sigma = C\e(u)$ in \eqref{ed-eq1}, and get a system of equations
\begin{align} \label{ed-velstress}
\rho \dot v - \div \sigma = f, \qquad A \dot \sigma = \e(v), 
\end{align}
where $A = C^{-1}$. For boundary conditions we take $v=0$, implied by the vanishing of $u$ on $\pd \Omega$, and, for initial conditions, $\sigma(0) =\sigma_0:= C \e(u_0)$, $v(0) = v_0$. We assume that the mass density $\rho$ satisfies $0 < \rho_0 \leq \rho \leq \rho_1 < \infty$ for constants $\rho_0$, $\rho_1$.

To establish well-posedness of this system, we recall the Hille--Yosida theorem. For a Hilbert space $\mathcal{X}$ and a closed, densely defined operator $\mathcal{L}$ on $\mathcal{X}$ with domain $D(\mathcal{L})$, we consider an evolution equation $\dot U = \mathcal{L} U + F$ with initial condition $U(0)=U_0$. The operator $\mathcal{L}$ is dissipative if $(\mathcal{L} u, u)_{\mathcal{X}}\le 0$ for $u \in D(\mathcal{L})$, and it is $m$-dissipative
if, further, $I - \mathcal{L} : D(\mathcal{L}) \ra \mathcal{X}$ is surjective (see \cite{Cazenave-Alain-book}, Definition~2.2.2, Proposition~2.2.6, and
Proposition~2.4.2). The Hille--Yosida theorem states that,
if $\mathcal{L}$ is an $m$-dissipative operator, $F \in W^{1,1}([0, T_0]; \mathcal{X})$, and  $U_0 \in D(\mathcal{L})$, then the initial value problem has a unique solution $U \in C^0([0,T_0]; D(\mathcal{L}) \cap C^1([0,T_0]; \mathcal{X})$ (see \cite{Cazenave-Alain-book}, Proposition 4.1.6). We now apply this to \eqref{ed-velstress}, which we rewrite as
\begin{align*}
\begin{pmatrix}
\dot \sigma \\
\dot v
\end{pmatrix}
=
\begin{pmatrix}
0 & C \e \\
\rho^{-1}\div & 0
\end{pmatrix}
\begin{pmatrix}
\sigma \\
v
\end{pmatrix}
+
\begin{pmatrix}
0 \\
\rho^{-1} f
\end{pmatrix}
.
\end{align*}
Let $\mathcal{X} = L^2(\Omega; \mathbb{S}) \times V$ be the Hilbert space with the inner product 
\begin{align*}
((\sigma, v), (\tau, w))_{\mathcal{X}} := (\sigma, \tau)_A + (v, w)_\rho = (A \sigma, \tau) + (\rho v, w). 
\end{align*}
We define the linear operator $\mathcal{L}$ as $\mathcal{L}(\sigma, v) = (C \e(v), \rho^{-1} \div \sigma)$. Note that $\mathcal{L}$ is an unbounded operator on $\mathcal{X}$ and its domain $D(\mathcal{L}) = S \times \mathring{H}^1(\Omega; \V)$ is dense in $\mathcal{X}$.

To apply the Hille--Yosida theorem, we verify that $\mathcal{L}$ is $m$-dissipative.  Let $(\sigma,v) \in D(\mathcal{L})$.  Then
\begin{align*}
(\mathcal{L} (\sigma,v) , (\sigma,v) )_{\mathcal{X}} 
&= ((C \e(v) , \rho^{-1}\div \sigma), (\sigma,v))_{\mathcal{X}} \\
&=  (\e(v) , \sigma) +(\div\sigma,v) = 0 ,
\end{align*}
where the last equality comes from  the integration by parts.  Thus $\mathcal L$ is dissipative.  To show that
it is $m$-dissipative, it remains to prove that $I - \mathcal{L}$ is surjective.  We shall show that,
for any given $(\eta, p) \in \mathcal{X}$, the weakly formulated problem
\begin{align*}
((I - \mathcal{L}) (\sigma, v), (\tau, w))_{\mathcal{X}} = ((\eta, p), (\tau, w))_{\mathcal{X}}, \qquad (\tau, w) \in D(\mathcal{L}),
\end{align*}
has a solution $(\sigma,v)\in D(\mathcal L)$.  If $(\sigma,v)$ satisfies
this weak formulation, then $(I - \mathcal{L}) (\sigma, v) = (\eta, p)$, since $D(\mathcal{L})$ is dense in $\mathcal{X}$.
The weak problem may be restated as
\begin{align}
\label{ed-hy-eq1} (\sigma - C\e(v), \tau)_A &= (\eta, \tau)_A, & &  \tau \in S, \\
\label{ed-hy-eq2} (v - \rho^{-1} \div \sigma, w)_\rho &= (p, w)_\rho, & &  w \in \mathring{H}^1(\Omega; \V). 
\end{align}
Rewriting \eqref{ed-hy-eq2} using the integration by parts and the symmetry of $\sigma$,  we get
\begin{align} \label{ed-hy-weak2}
(\rho v, w) + (\sigma, \e(w)) &= (\rho p, w), \quad w \in \mathring{H}^1(\Omega; \V). 
\end{align}
The equation \eqref{ed-hy-eq1} gives a constraint $\sigma - C\e(v) = \eta$, and substituting $\sigma$ in \eqref{ed-hy-weak2} by $C \e(v) + \eta$, we obtain
\begin{align*}
(\rho v, w) + (C \e(v), \e(w)) = (\rho p, w) - (\eta, \e(w)), \quad w \in \mathring{H}^1(\Omega; \V).
\end{align*}
By Korn's inequality and the Lax--Milgram lemma, this equation has a unique solution $v \in \mathring{H}^1(\Omega; \V)$. One can easily see that $\sigma= C \e(v) + \eta$ is in $L^2(\Omega; \mathbb{S})$, and also in $M \cap L^2(\Omega; \mathbb{S}) = S$ because the equation \eqref{ed-hy-weak2} implies that $\div \sigma$ is well-defined in the sense of distributions. This completes the verification that $\mathcal{L}$ is $m$-dissipative.

We may therefore apply the Hille--Yosida theorem, and obtain the following result.
Given $\sigma_0 \in S$, $v_0 \in \mathring{H}^1(\Omega; \V)$, and $f \in W^{1,1}([0,T_0]; \V)$, then there exist
\begin{align*}
\sigma &\in C^0([0,T_0]; S) \cap C^1([0,T_0]; L^2(\Omega; \mathbb{S})), \\
v &\in C^0([0,T_0]; \mathring{H}^1(\Omega; \V)) \cap C^1([0,T_0]; V),
\end{align*}
satisfying the evolution equations \eqref{ed-velstress} and assuming the given initial data. 

Now we describe a weak formulation of \eqref{ed-velstress} with weak symmetry of stress. We assume that $\sigma_0 = C \e(u_0)$ for some $u_0 \in \mathring{H}^1(\Omega; \V)$. If we define 
\begin{align} \label{u-integral}
u(t) = u_0 + \int_0^t v(s)\,ds,
\end{align}
then, using $A \dot \sigma = \e(v)$ and the fundamental theorem of calculus, we get $A \sigma = \e(u)$. If we set 
\begin{align} \label{r-def}
r = \skw \grad u,
\end{align}
then  $\dot r = \skw \grad v$. Integrating the second equation of \eqref{ed-velstress} by parts with the boundary conditions $v \equiv 0$ on $\pd \Omega$, we get $(A \dot \sigma, \tau) = (\e(v), \tau) = (\grad v - \dot r, \tau) = -(v, \div \tau) - (\dot r, \tau)$ for all $\tau \in M$, i.e.,
\begin{align}
(A \dot{\sigma}, \tau) + ( \div \tau, v ) + (\dot r, \tau) = 0 , \qquad \tau \in M.
\end{align}
From the first equation of \eqref{ed-velstress}, we get $( \rho \dot{v}, w ) - ( \div \sigma, w ) = ( f, w )$ for $w \in V$. Finally, the symmetry of $\sigma$ gives $(\dot \sigma, q) = 0$ for $q \in K$. The equations together constitute
our weak formulation with weak symmetry of stress.  We seek
\begin{equation}\label{ed-regular}
 \begin{gathered}
\sigma \in C^0 ([0, T_0]; M) \cap C^1 ([0, T_0]; L^2(\Omega; \mathbb{M})), \\
 v \in C^1 ([0, T_0];V)) , \quad   r \in C^1  ([0, T_0];K),
\end{gathered}
\end{equation}
such that
\begin{align}
\label{ed-weaksym1} (A \dot{\sigma}, \tau) + ( \div \tau, v ) + (\dot r, \tau) &= 0 , \qquad \qquad \tau \in M, \\
\label{ed-weaksym2} ( \rho \dot{v}, w ) - ( \div \sigma, w ) &= ( f, w ) , \qquad w \in V,\\
\label{ed-weaksym3} (\dot \sigma, q) &= 0 , \qquad \qquad  q \in K,  
\end{align}
with given initial data $(\sigma_0, v_0, r_0) = (C\e(u_0), v_0, \skw \grad u_0)$.
We now show that this problem is well-posed.
\begin{thm} Let $f\in W^{1,1}([0,T_0];\V)$ and $u_0, v_0 \in \mathring{H}^1(\Omega; \V)$.  Set $\sigma_0 = C\e(u_0)$, $r_0 = \skw \grad u_0$.  Then the system {\rm (\ref{ed-regular}--\ref{ed-weaksym3})} has a unique solution assuming the given
initial data.
\end{thm}
\begin{proof} By the Hille--Yosida theorem, the equation \eqref{ed-velstress} has a solution $(\sigma, v)$ with the initial data $(\sigma_0, v_0)$. We define $u$ by \eqref{u-integral} and $r$ by \eqref{r-def}.  The resulting triple $(\sigma, v, r)$ then satisfies (\ref{ed-weaksym1}--\ref{ed-weaksym3}), and takes on the desired initial values.  We have thus
proven existence of a solution.

For uniqueness, suppose that there are two solutions of (\ref{ed-weaksym1}--\ref{ed-weaksym3}) with same initial data, and denote their difference by $(\sigma^d, v^d, r^d)$. Then this triple satisfies 
\begin{align}
\label{ed-weaksym1-diff} (A \dot{\sigma}^d, \tau) + ( \div \tau, v^d ) + (\dot r^d, \tau) &= 0 , & & \tau \in M, \\
\notag ( \rho \dot{v}^d, w ) - ( \div \sigma^d, w ) &= 0 , & & w \in V,\\
\notag (\dot \sigma^d, q) &= 0 , & & q \in K, 
\end{align}
with $\sigma^d(0)$,  $v^d(0)$, and $r^d(0)$ all zero. Now we set $\tau = \sigma^d$, $w = v^d$ in the first two equations and add them. Since $\sigma^d \perp K$ and $\dot r \in K$, we have $(\dot r, \sigma) = 0$, so the sum of two equations gives 
\begin{align*}
\half \frac d{dt} \| \sigma^d \|_A^2 + \half \frac d{dt} \| v^d \|_\rho^2 = 0.
\end{align*}
Therefore $\| \sigma^d (t) \|_A^2 + \| v^d (t) \|_\rho^2 = \| \sigma^d (0) \|_A^2 + \| v^d (0) \|_\rho^2 =0$, so $\sigma^d \equiv 0 \equiv v^d$. From \eqref{ed-weaksym1-diff}, one then sees that $\dot r^d \equiv 0$ as well. Since $r^d(0) = 0$, we have $r^d \equiv 0$, so uniqueness is proved. 
\end{proof}

We close this section by pointing out the straightforward changes needed to handle mixed displacement--traction boundary
conditions in our velocity-stress formulation. Suppose $\Gamma_D$ and $\Gamma_N$ are two disjoint open subsets of
$\pd\Omega$ with $\pd \Omega = \overline {\Gamma}_D \cup \overline{\Gamma}_N$ and $\Gamma_D$ nonempty. We consider the boundary conditions $v = g$ on $\Gamma_D$, $\sigma \nu = \kappa$ on $\Gamma_N$ where
\begin{align*}
g : [0, T_0] \times \Gamma_D \ra \R^n, \qquad \kappa : [0, T_0] \times \Gamma_N \ra \R^n
\end{align*}
are given.
We define $M_{\Gamma_N} = \{ \tau \in M \,|\, \tau \nu = 0 \text{ on } \Gamma_N \}$.
Then a velocity-stress formulation with weak symmetry seeks $(\sigma, v, r)$ satisfying \eqref{ed-regular} with $\sigma \nu = \kappa$ on $\Gamma_N$ and 
\begin{align} 
\notag (A \dot{\sigma}, \tau) + ( \div \tau, v ) + (\dot r, \tau) &= \int_{\Gamma_D} g \cdot \tau \nu \, ds , & &\tau \in M_{\Gamma_N}, \\
\label{ed-inhomog-dirichlet} ( \rho \dot{v}, w ) - ( \div \sigma, w ) &= ( f, w ) , & & w \in V , \\ 
\notag (\dot \sigma, q) &= 0 , & & q \in K.
\end{align}
The initial data must satisfy the compatibility conditions $\sigma_0 \nu = \kappa(0)$ on $\Gamma_N$ and $v_0 = g(0)$ on $\Gamma_D$.

\section{Semidiscrete error analysis for the AFW and CGG elements} \label{section:semidiscrete}
In this section we consider spatial discretization of problem (\ref{ed-weaksym1}--\ref{ed-weaksym3}) with given initial data. We show existence and uniqueness of semidiscrete solutions and discuss the semidiscrete error analysis. Although the main result of this section is stated for the AFW and CGG elements, the results in this section are valid for all elements in Table \ref{mixed-approx}. We will discuss improved results for the Stenberg and GG elements in section~\ref{section:imp-err-analysis}.

\subsection{The semidiscrete problem} Let $M_h \times V_h \times K_h$ be one of the elements in {\rm Table} \ref{mixed-approx}.
Given initial data $(\sigma_{h0}, v_{h0}, r_{h0})\in M_h \times V_h \times K_h$, the semidiscretization of {\rm (\ref{ed-weaksym1}--\ref{ed-weaksym3})} seeks
\begin{equation}\label{ed-semi-regular}
 \sigma_h \in C^1 ([0, T_0]; M_h), \quad v_h \in C^1 ([0, T_0];V_h), \quad r_h \in C^1 ([0, T_0];K_h),
\end{equation}
satisfying the equations
\begin{align} 
\label{ed-semidiscrete-eq1} (A \dot{\sigma}_h, \tau) + (\div \tau, v_h) + ({\dot r}_h, \tau) &= 0, & & \tau \in M_h, \\
\label{ed-semidiscrete-eq2} ( \rho  \dot{v}_h, w ) - (\div  \sigma_h, w)  &= ( f, w ), & & w \in V_h,  \\
\label{ed-semidiscrete-eq3} (\dot \sigma_h, q ) &= 0, & & q \in K_h, 
\end{align}
for all time $t \in [0, T_0]$, and assuming the given initial data.

\begin{thm}
The semidiscrete system has a unique solution.
\end{thm}
\begin{proof} Let $\{ \phi_i \}$, $\{ \psi_i \}$, $\{ \chi_i \}$ be bases of $M_h$, $V_h$, and $K_h$, respectively. We use $\mathscr{A}$, $\mathscr{B}$, $\mathscr{C}$, $\mathscr{M}$ to denote the matrices whose $(i,j)$-entries are
\begin{align*}
(A \phi_j, \phi_i), \quad (\div \phi_j, \psi_i), \quad (\phi_j, \chi_i), \quad (\rho \psi_j, \psi_i),
\end{align*}
respectively. We write $\sigma_h = \sum_i \alpha_i \phi_i$, $v_h = \sum_i \beta_i \psi_i$, $r_h = \sum_i \gamma_i \chi_i$, and set $\zeta_i =(f,\psi_i)$, and use $\alpha$, $\beta$, $\gamma$, $\zeta$ to denote the corresponding vectors. Then we may rewrite (\ref{ed-semidiscrete-eq1}--\ref{ed-semidiscrete-eq3}) in a matrix equation form, 
\begin{align*} 
\begin{pmatrix}
\mathscr{A} & 0 & \mathscr{C}^T \\
0 & \mathscr{M} & 0 \\
\mathscr{C} & 0 & 0
\end{pmatrix}
\begin{pmatrix}
\dot \alpha \\
\dot \beta \\
\dot \gamma
\end{pmatrix}
= 
\begin{pmatrix}
0 & -\mathscr{B}^T & 0 \\
\mathscr{B} & 0 & 0 \\
0 & 0 & 0
\end{pmatrix}
\begin{pmatrix}
\alpha \\
\beta \\
\gamma
\end{pmatrix}
+ 
\begin{pmatrix}
0\\
\zeta \\
0
\end{pmatrix}.
\end{align*}
The above matrix equation is a linear system of ordinary differential equations. Note that the coefficient matrix on the left-hand side is invertible because $\mathscr{A}$ and $\mathscr{M}$ are positive definite and $\mathscr{C}^T$ is injective from the inf-sup condition {\bf (A1)}. By standard ODE theory (see \cite{Coddington-Levinson-book}, p.75), the matrix equation is well-posed as an initial value problem, so the existence and uniqueness of solutions of (\ref{ed-semidiscrete-eq1}--\ref{ed-semidiscrete-eq3}) follow.
\end{proof}

Next we discuss the construction of initial data for the semidiscretization, starting from the initial
data $u_0,v_0\in \mathring{H}^1(\Omega; \V)$ for the continuous problem.
As initial data for the velocity we simply take
\begin{equation}\label{did0}
 v_{h0}=P_h v_0.
\end{equation}
Recall that we obtained initial
data for $\sigma$ and $r$ as $\sigma_0 = C\e(u_0)$ and $r_0 = \skw \grad u_0$.  Consequently,
$(A \sigma_0, \tau) + (\div \tau, u_0) + (r_0, \tau) = 0$ for $\tau \in M$, and $\sigma_0\perp q$ for $q\in K$.  We compute the initial data
for $\sigma_h$, $u_h$, and $r_h$, from
a mixed elliptic problem:  $(\sigma_{h0}, u_{h0}, r_{h0}) \in M_h \times V_h \times K_h$ of the system,
\begin{align}
\label{did1}
(A \sigma_{h0}, \tau) + (\div \tau, u_{h0}) + (r_{h0}, \tau) &= 0, & & \tau \in M_h, \\
\label{did2}
(\div \sigma_{h0}, w) &= (\div \sigma_0, w), & & w \in V_h, \\
\label{did3}
(\sigma_{h0}, q) &= 0,  & & q \in K_h,
\end{align}
for which we know, by section~\ref{section:mfem}, that there exists a unique solution and we have the error estimate
\begin{equation}
 \label{ed-compatible-ic2}
\| \sigma_0 - \sigma_{h0}, r_0 - r_{h0} \| \leq ch^m \| \sigma_0, r_0 \|_m, \qquad 1 \leq m \leq 
\begin{cases}
 k, & \text{AFW, CGG},\\ k+1, & \text{Stenberg, GG}.
\end{cases}
\end{equation}

\subsection{Decomposition of semidiscrete errors} For the error analysis, we follow a standard approach: representatives of $(\sigma, v, r)$ are used to split the semidiscrete error into the projection error and the approximation error, and bounds are obtained by a priori error analysis. 

We now state the main convergence result for the AFW and CGG elements.
\begin{thm} \label{ed-semidiscrete-thm} Let $(M_h, V_h, K_h)$ be the \rm{AFW} or \rm{CGG} elements in {\rm Table \ref{mixed-approx}} of order $k \geq 1$ and let $m$ be a real number such that $1 \leq m \leq k$. Suppose that $\sigma, v, r \in W^{1,1} ([0,T_0]; H^m)$ and let $(\sigma_h, v_h, r_h)$ be the solution of {\rm (\ref{ed-semi-regular}--\ref{ed-semidiscrete-eq3})} with initial data $(\sigma_{h0}, v_{h0}, r_{h0})$ defined as in (\ref{did0}--\ref{did3}). Then we have
\begin{gather*}
\| \sigma - \sigma_h , v - v_h , r - r_h \|_{L^\infty L^2} \leq c h^m \| \sigma, v, r \|_{W^{1,1} H^m},
\end{gather*}
where $c$ depends on the compliance tensor $A$, and the lower and upper bounds of the mass density $\rho_0$, $\rho_1$. 
\end{thm}

For our error analysis, we denote the semidiscrete errors, i.e., the difference of the exact solution $(\sigma, v, r)$ and the semidiscrete solution $(\sigma_h, v_h, r_h)$, by 
\begin{align*}
\es = \sigma - \sigma_h, \quad e_v = v - v_h, \quad e_r = r - r_h.
\end{align*}
Then, by taking differences of equations (\ref{ed-weaksym1}--\ref{ed-weaksym3}) and (\ref{ed-semidiscrete-eq1}--\ref{ed-semidiscrete-eq3}), we get  
\begin{align}
\label{ed-err-eq1}  (A \dot e_\sigma, \tau) + (\div \tau, e_v) + (\dot e_r, \tau) &= 0 , & & \tau \in M_h, \\
\label{ed-err-eq2}  ( \rho \dot e_v , w ) - (\div \es, w)  &= 0 , & & w \in V_h, \\
\label{ed-err-eq3} (\dot e_\sigma, q) &= 0, & & q \in K_h.
\end{align}
Recall that $\tilde \Pi_h$ is the weakly symmetric elliptic projection in Lemma \ref{weak-sym-elliptic} and $P_h$, $P_h'$ are the orthogonal $L^2$ projections onto $V_h$ and $K_h$, respectively. We decompose the semidiscrete errors $(\es , e_v , e_r)$ into
\begin{align}
\label{ed-err-decomp1} \es &= \es^P + \es^h := (\sigma - \tilde \Pi_h \sigma) + (\tilde \Pi_h \sigma - \sigma_h) ,\\
\label{ed-err-decomp2} e_v &= e_v^P + e_v^h := (v-P_h v) + (P_h v - v_h) ,\\
\label{ed-err-decomp3} e_r &= e_r^P + e_r^h := (r - P_h' r) + (P_h' r - r_h).
\end{align}
We call the $e^P$ terms the projection errors and the $e^h$ terms the approximation errors, respectively.
We shall prove Theorem \ref{ed-semidiscrete-thm} by bounding the projection errors in section~\ref{semi-proj-estm} and the approximation errors in section \ref{semi-appr-estm}.
First, we remark that
\begin{align} \label{ed-proj-err-prop}
\begin{split}
(\div \tau, e_v^P)  &= 0, \qquad \tau \in M_h, \\
(\div e_\sigma^P, w) &= 0, \qquad w \in V_h,
\end{split}
\end{align}
as follows from {\bf (A0)} in section~\ref{section:mfem} and \eqref{es-weaksym-proj1}.

\subsection{Projection error estimates for the AFW and CGG elements} \label{semi-proj-estm}
A priori estimates of the $L^\infty L^2$ norms of the projection errors follow from the approximability of $M_h \times V_h \times K_h$.  
\begin{thm} \label{ed-proj-thm} There exists a constant $c>0$ such that  
\begin{align}
 \label{ed-proj-err1} \| \es^P \| &\leq ch^m \| \sigma \|_{H^m} , \quad   1 \leq m \leq k, \\
\label{ed-proj-err2} \| e_v^P  \| &\leq c h^m \| v \|_{H^m}  , \quad   0 \leq m \leq k,  \\
\label{ed-proj-err3} \| e_r^P  \| &\leq c h^m \| r \|_{H^m}  , \quad  0 \leq m \leq k ,
\end{align}
at each time $t\in[0,T_0]$.
Furthermore, similar inequalities hold with $\sigma$, $v$, and $r$, replaced by their time derivatives.
\end{thm}
\begin{proof} For any $t \in [0,T_0]$ and $1\leq m \leq k$, by \eqref{es-weaksym-proj2} and \eqref{piest}, we have
\begin{align*}
\| \es^P (t) \| = \| \sigma(t) - \tilde \Pi_h \sigma (t)\| \leq c\| \sigma(t) - \Pi_h \sigma (t)\| \leq ch^m \| \sigma(t) \|_{m},
\end{align*}
and \eqref{ed-proj-err1} is proved. Similarly, from definitions of $e_v^P$ and $e_r^P$, we have 
\begin{align*}
\| e_v^P (t) \| \leq ch^m \| v (t) \|_{m}, \qquad \| e_r^P (t) \| \leq ch^m \| r(t) \|_{m}, 
\end{align*}
for any $t \in [0,T_0]$, $0 \leq m \leq k$. The same argument applies to time derivatives of the projection errors because the projections
$\tilde \Pi_h$, $P_h$, $P_h'$ commute with time differentiation.
\end{proof}

\subsection{Approximation error estimates for the AFW and CGG elements} \label{semi-appr-estm}
Now we estimate the $L^\infty L^2$ norms of the approximation errors.
\begin{thm} \label{ed-approx-thm} For $1 \leq m \leq k$, 
\begin{align}
 \| \es^h, e_v^h, e_r^h \|_{L^\infty L^2} \leq ch^m \| \sigma, v, r \|_{W^{1,1} H^m},
\end{align}
where $c$ depends on $\rho_0$, $\rho_1$, and $A$.
\end{thm}
\begin{proof} The proof is based on two estimates: 
\begin{align} \label{ed-approx-estm1}
\| \es^h , e_v^h \|_{L^\infty L^2} &\leq ch^m (\| \sigma_0, r_0 \|_m + \| \dot \sigma, \dot v, \dot r \|_{L^1 H^m}), \\
\label{ed-approx-estm2} \| e_r^h \|_{L^\infty L^2} &\leq c\| e_\sigma^h, e_\sigma^P, e_r^P \|_{L^\infty L^2}, 
\end{align}
for $1 \leq m \leq k$. Theorem \ref{ed-approx-thm} follows from these estimates
and Theorem~\ref{ed-proj-thm}, since $\| \sigma, r \|_{L^\infty H^m} \leq c \| \sigma, r \|_{W^{1,1}H^m}$ by Sobolev embedding.

To prove (\ref{ed-approx-estm1}--\ref{ed-approx-estm2}), we first remark that $\sigma_{h0} \perp K_h$ from \eqref{did3} and $\tilde \Pi_h \sigma_0 \perp K_h$ from the definition of $\tilde\Pi_h$, and so $e_\sigma^h(0) \perp K_h$. Similarly, $\dot\sigma_h\perp K_h$
from \eqref{ed-semidiscrete-eq3} and $\tilde \Pi_h\dot\sigma\perp K_h$ from the definition
of $\tilde\Pi_h$ and its commutativity with time differentiation, so $\dot e_\sigma^h\perp K_h$. Combining these facts, we deduce that $e_\sigma^h\perp K_h$, as well.
To show \eqref{ed-approx-estm1}, we rewrite (\ref{ed-err-eq1}--\ref{ed-err-eq2}), using the notations in (\ref{ed-err-decomp1}--\ref{ed-err-decomp3}) and the reductions in \eqref{ed-proj-err-prop}, as 
\begin{align}
\label{ed-semi-eq1} (A \dot e_\sigma^h, \tau) + (\div \tau, e_v^h) + (\dot e_r^h, \tau) &= - (A \dot e_\sigma^P, \tau) - (\dot e_r^P, \tau) , & \tau \in M_h, \\
\label{ed-semi-eq2} (\rho \dot e_v^h , w ) - (\div \es^h , w) &= - (\rho \dot e_v^P , w )  , & w \in V_h .
\end{align}
We take $\tau = \es^h$, $w = e_v^h$ in the above two equations, add them, and use the fact $\es^h \perp \dot e_r^h$ from $\dot e_r^h \in K_h$, obtaining
\begin{align}  \label{ed-evol-eq0}
\half \frac{d}{dt} \| \es^h \|_A^2 + \frac 12 \frac{d}{dt} \| e_v^h \|_{\rho}^2 = - (A \dot e_\sigma^P, \es^h ) - (\dot e_r^P, \es^h ) - (\rho \dot e_v^P, e_v^h ).
\end{align}
Bounding the right-hand side of this inequality using the Cauchy--Schwarz inequality and
the bounds on $A$ and $\rho$, we get
\begin{align}
\label{ed-evol-eq1}\half \frac{d}{dt} (\| \es^h \|_A^2 + \| e_v^h \|_{\rho}^2) \leq  c \| \dot e_\sigma^P, \dot e_r^P, \dot e_v^P \| \,(\| \es^h \|_A^2 + \| e_v^h \|_{\rho}^2)^\half.
\end{align}
Dividing both sides by $(\| \es^h \|_A^2 + \| e_v^h \|_{\rho}^2)^{1/2}$ and integrating in time on $[0,t]$, then
\begin{align}  \label{ed-evol-eq2}
\left ( \| \es^h (t) \|_A^2 + \| e_v^h (t)  \|_\rho^2 \right)^\half \leq \left (\| \es^h(0) \|_A^2 + \| e_v^h(0) \|_\rho^2 \right)^\half + c\int_0^t \| \dot e_\sigma^P, \dot e_r^P, \dot e_v^P \| \,ds .
\end{align}
Since $A$ is coercive and $\rho$ has a positive lower bound, in order to establish \eqref{ed-approx-estm1}, it suffices to show that the right-hand side of \eqref{ed-evol-eq2} is bounded by $ch^m (\| \sigma_0, r_0 \|_m + \| \dot \sigma, \dot v, \dot r \|_{L^1 H^m})$. For the integral term, this follows directly from Theorem~\ref{ed-proj-thm}. We also have that $e_v^h (0) = 0$, from the choice of $v_{h0}$. Finally, we use the boundedness of $A$ and $\rho$, the triangle inequality, \eqref{ed-compatible-ic2}, \eqref{es-weaksym-proj2}, \eqref{piest}, to get
\begin{align} \label{sigma0-estm}
\| \es^h(0)\|_A &\leq c (\|\sigma_{h0} - \sigma_0\| + \|\sigma_0 - \tilde \Pi_h \sigma_0\|) \leq ch^m \| \sigma_0, r_0 \|_m, 
\end{align}
for $1 \leq m \leq k$.  This completes the proof of \eqref{ed-approx-estm1}.

To complete the proof of the theorem, we now verify \eqref{ed-approx-estm2}.  Since
$\sigma=C\epsilon(u)$ and $r=\skw\grad u$, we have
\begin{align*}
(A \sigma, \tau) + (u, \div \tau) + (r, \tau) = 0, \qquad \tau \in M,
\end{align*}
and therefore $(A \sigma, \tau) + (r , \tau) = 0$ for $\tau\in M$ divergence-free.
Similarly, from \eqref{did1} we have $(A \sigma_{h0}, \tau) + (r_{h0}, \tau) = 0$
for $\tau\in M_h$ divergence-free.  Subtracting, we see that $(A e_\sigma(0) , \tau) + (e_r(0) , \tau) = 0$ for such $\tau$.  Next, we may take a divergence-free $\tau$
in \eqref{ed-err-eq1} for $\tau \in M_h$ to find that $(A\dot e_\sigma, \tau) + (\dot e_r, \tau)$ vanishes as well.  Combining, we conclude that
\begin{align*}
(A e_\sigma, \tau) + (e_r , \tau) = 0,\quad \tau \in M_h, \; \div \tau = 0,
\end{align*}
or, equivalently,
\begin{align*}
(e_r^h, \tau) = - (A (e_\sigma^h+ e_\sigma^P) , \tau) + (e_r^P,\tau), \quad \tau \in M_h, \; \div \tau = 0,
\end{align*}
for all $t\in [0,T_0]$.  Now fix $t$, and choose $\tau \in M_h$ such that $\div \tau = 0$, $(\tau, e_r^h (t)) = \| e_r^h (t) \|^2$, and $\| \tau \| \leq c \| e_r^h(t) \|$, which is possible by the stability condition {\bf (A1)} in section~\ref{section:stationary}.  It follows that $\| e_r^h(t) \| \leq c \| e_\sigma^h (t), e_\sigma^P(t), e_r^P (t) \|$, from which \eqref{ed-approx-estm2} follows.
\end{proof}
Combining Theorems \ref{ed-proj-thm} and \ref{ed-approx-thm}, we complete the proof of the Theorem \ref{ed-semidiscrete-thm}, the error estimates for the semidiscrete solutions.

\subsection{Robustness for nearly incompressible materials} \label{ed-robust}

Throughout this section, we assume that the elastic medium is homogeneous and isotropic, i.e., the compliance tensor $A$ has the form \eqref{compliance} with Lam\'e coefficients $\mu$ and $\lambda$ which are constant. We continue to consider homogeneous displacement boundary conditions. In nearly incompressible elastic materials, $\lambda$ is very large, and, in the incompressible limit, $\lambda = +\infty$. Many standard discretization of elasticity suffer from locking, which means that the errors, while they decay with the mesh size, grow as $\lambda$ increases. A robust or locking-free method is one in which the error estimates hold uniformly as $\lambda \ra +\infty$. In contrast to many
displacement methods, mixed methods for stationary elasticity problems are typically locking-free (see \cite{ADG84, BFBook}).
In this section, we show that our mixed method for linear elastodynamics is likewise free of locking. Again, we focus on semidiscretization in space, which is the essential aspect. For an analysis taking into account temporal discretization, we refer \cite{LeeThesis}.

We require the following lemmas, proved in \cite{ADG84}. Let $\tau^D := \tau - (1/n) \tr(\tau) I$ denote the deviatoric part of $\tau$ in $L^2(\Omega; \M)$.
\begin{lemma} \label{devi-lemma1}
Let $\tau \in M$ satisfy $\int_{\Omega} \tr(\tau) \, dx = 0$. Then the estimate  
\begin{align}
\label{devi-estm1} \| \tau \| \leq c (\| \tau^D \| + \| \div \tau \|_{-1} ), 
\end{align}
holds with $c >0$ independent of $\tau$. 
\end{lemma}
\begin{lemma} \label{devi-lemma2}
For $\tau \in L^2(\Omega; \M)$ and $A$ of the form in \eqref{compliance}, the inequality
\begin{align}
\label{devi-estm2} \| \tau^D \|^2 \leq c\| \tau \|_A^2 , 
\end{align}
holds with $c$ depending only on $\mu$ and $n$. 
\end{lemma}
\begin{thm} Let $M_h \times V_h \times K_h$ be one of the elements in {\rm Table \ref{mixed-approx}} of order $k \geq 1$ and assume that $A$ has the form of \eqref{compliance} with $\mu$ and $\lambda$ constant. We assume that the exact solution $\sigma$, $v$, and $r$ belong to $W^{2,1}H^k$. Then there exist a constant $c>0$ independent of $\lambda$ such that
\begin{align} 
\label{locking-free-estm1} \| v - v_h \|_{L^\infty L^2}  &\leq ch^k \| \sigma, v, r \|_{W^{1,1} H^k}, \\
\label{locking-free-estm2} \| \sigma - \sigma_h \|_{L^\infty L^2} &\leq ch^k \| \sigma, v, r \|_{W^{2,1} H^k}.
\end{align}
\end{thm}
\begin{proof}
The projection error estimates in Theorem \ref{ed-proj-thm} certainly hold with a constant
$c$ independent of $\lambda$, 
because $\tilde \Pi_h$, $P_h$, $P_h'$ do not depend on $\lambda$. Furthermore, the inequality $\| e_\sigma^P \|_A \leq c \| e_\sigma^P \|$ holds uniformly in $\lambda$, since $A$ remains uniformly bounded as $\lambda \ra + \infty$. 

The proof is based on the following estimates, in which the constant
$c$ does not depend on $\lambda$: 
\begin{align} 
\label{locking-inter-estm1} (\| \es^h (t) \|_A^2 + \| e_v^h (t) \|_{\rho}^2)^\half &\leq c h^k \| \sigma, v, r \|_{W^{1,1} H^k}, \\
\label{locking-inter-estm2} \| e_\sigma (t) \| &\leq c(\| e_\sigma (t) \|_A + \| \div e_\sigma (t) \|_{-1}), \\
\label{locking-inter-estm3} \| \div e_{\sigma}(t) \|_{-1} &\leq c (\| \dot e_v (t) \| + h^k \| \sigma (t) \|_k), \\
\label{locking-inter-estm4} \| \dot e_v^h (t) \| &\leq ch^k \| \sigma, v, r \|_{W^{2,1} H^k}.
\end{align}
We first show that \eqref{locking-free-estm1} and \eqref{locking-free-estm2} follow from these estimates. The estimate \eqref{locking-free-estm1} is a consequence of \eqref{locking-inter-estm1}, the estimate on $\| e_v^P \|$ in Theorem \ref{ed-proj-thm}, and the triangle inequality. To show \eqref{locking-free-estm2}, observe that \eqref{locking-inter-estm2}, \eqref{locking-inter-estm3}, and the triangle inequality give
\begin{align*}
\| e_\sigma (t) \| \leq c( \| e_\sigma^h (t) \|_A + \| e_\sigma^P (t) \|_A + \| \dot e_v^h (t) \| + \| \dot e_v^P (t) \| + h^k \| \sigma(t) \|_k).
\end{align*}
Then \eqref{locking-free-estm2} is obtained by \eqref{locking-inter-estm1}, \eqref{locking-inter-estm4}, and Theorem \ref{ed-proj-thm}. 

To prove \eqref{locking-inter-estm1}, observe that $\tr(\dot e_r^P) = 0$ because $\dot e_r^P$ is skew-symmetric, so $\dot e_r^P = A(2 \mu \dot e_r^P)$ holds for $A$ of the form \eqref{compliance}. We may therefore rewrite \eqref{ed-evol-eq0} as
\begin{align*}
\half \frac{d}{dt} (\| \es^h \|_A^2 + \| e_v^h \|_{\rho}^2) = - (A (\dot e_\sigma^P + 2 \mu \dot e_r^P), \es^h ) - (\rho \dot e_v^P, e_v^h ),
\end{align*}
and repeating the argument in (\ref{ed-evol-eq0}--\ref{ed-evol-eq2}), we have 
\begin{multline*}
(\| \es^h (t) \|_A^2 + \| e_v^h (t) \|_{\rho}^2)^\half \\
\leq (\| \es^h (0) \|_A^2 + \| e_v^h (0) \|_{\rho}^2)^\half+
\int_0^t (\|\dot e_\sigma^P+2\mu \dot e_r^P\|_A^2 + \| \dot e_v^P\|_\rho^2)^\half\,ds.
\end{multline*}
Since $A$ is uniformly bounded in $\lambda$, \eqref{locking-inter-estm1} follows from \eqref{sigma0-estm}, $e_v^h(0) = 0$, Theorem \ref{ed-proj-thm}, and Sobolev embedding $\| \sigma, r \|_{L^\infty H^k} \leq c \| \sigma, r \|_{W^{1,1}H^k}$.

To show \eqref{locking-inter-estm2}, by Lemma \ref{devi-lemma1} and Lemma \ref{devi-lemma2}, it is enough to show that $\int_{\Omega} \tr(e_\sigma (t)) \,dx = 0$. For $\tau = I$ in \eqref{ed-err-eq1}, satisfying $\div \tau = 0$ and $(\dot e_r, \tau) = 0$ due to the skew-symmetry of $\dot e_r$, we have $(A \dot e_\sigma (t), I) = 0$ for $t \in [0,T_0]$. From \eqref{did1} we see that $(A \es (0), I) = 0$, whence $(A \es (t), I) = 0$ for all $t \in [0, T_0]$. By the form of $A$ in \eqref{compliance}, 
\begin{align*} 
\int_{\Omega} \tr(e_\sigma (t)) \,dx = (e_\sigma (t), I) = (2\mu + n\lambda) (A e_\sigma (t), I) = 0.
\end{align*}

For \eqref{locking-inter-estm3}, by the triangle inequality, 
\begin{align*}
\| \div e_\sigma (t) \|_{-1} \leq \| \div e_\sigma^h (t) \|_{-1} + \| \div e_\sigma^P (t) \|_{-1} \leq \| \div e_\sigma^h (t) \| + \| \div e_\sigma^P (t) \|_{-1}, 
\end{align*}
so we only estimate $\| \div e_\sigma^h(t) \|$ and $\| \div e_\sigma^P (t) \|_{-1}$, separately. In \eqref{ed-semi-eq2}, $\div e_\sigma^h (t) = P_h (\rho \dot e_v(t))$, so $\| \div e_\sigma^h (t) \| \leq c \| \dot e_v (t) \|$. For the estimate of $\| \div e_\sigma^P (t) \|_{-1}$ it is enough to show $\| \div e_\sigma^P (t) \|_{-1} \leq ch \| \div e_\sigma^P (t) \|$ because 
\begin{align*}
\| \div e_\sigma^P (t) \| = \| \div \sigma(t) - P_h \div \sigma(t) \| \leq ch^{k-1} \| \sigma(t) \|_{k}, \qquad k \geq 1.
\end{align*}
For $w \in \mathring{H}^1(\Omega; \V)$ let $\bar w$ denote the $L^2$-orthogonal projection of $w$ into the space of $\V$-valued piecewise constant functions associated to the triangulation $\mathcal{T}_h$. By the definition of $\| \cdot \|_{-1}$ norm and the orthogonality $\div e_\sigma^P \perp V_h$, 
\begin{align*}
\| \div e_\sigma^P (t) \|_{-1} &= \sup_{w \in \mathring{H}^1(\Omega; \V)} \frac{(\div e_\sigma^P (t) , w)}{\| w \|_1} = \sup_{w \in \mathring{H}^1(\Omega; \V)} \frac{(\div e_\sigma^P (t) , w - \bar w)}{\| w \|_1} .
\end{align*}
By the Cauchy--Schwarz and the Poincar\'{e} inequalities, 
\begin{align*}
|(\div e_\sigma^P (t) , w - \bar w)| \leq ch \| \div e_\sigma^P (t) \| \| w \|_1
\end{align*}
holds and it gives $\| \div e_\sigma^P (t) \|_{-1} \leq ch \| \div e_\sigma^P (t) \|$ with the previous identity. 

For \eqref{locking-inter-estm4} we will show a stronger result which is similar to \eqref{locking-inter-estm1} for $\dot e_\sigma^h$ and $\dot e_v^h$. If we use the energy estimate argument, presented in (\ref{ed-semi-eq1}--\ref{ed-evol-eq2}), for time derivatives of \eqref{ed-semi-eq1} and \eqref{ed-semi-eq2} with $\tau = \dot e_\sigma^h$ and $w = \dot e_v^h$, then 
\begin{multline*}
(\| \dot e_\sigma^h (t) \|_A^2 + \| \dot e_v^h (t) \|_{\rho}^2)^\half \leq (\| \dot e_\sigma^h (0) \|_A^2 + \| \dot e_v^h (0) \|_{\rho}^2)^\half+ c\int_0^t (\|\ddot e_\sigma^P, \ddot e_r^P\|_A^2 + \| \ddot e_v^P\|_\rho^2)^\half\,ds.
\end{multline*}
The integral term is handled by Theorem \ref{ed-proj-thm} with $ch^k\| \sigma, v, r \|_{W^{2,1} H^k}$. To estimate $(\| \dot e_\sigma^h (0) \|_A^2 + \| \dot e_v^h (0) \|_\rho^2)^{1/2}$, take $t= 0$ in \eqref{ed-semi-eq1}, \eqref{ed-semi-eq2}, and use $A(2 \mu \dot e_r^P(0)) = \dot e_r^P(0)$ to have
\begin{align*}
(A \dot e_\sigma^h (0), \tau) + (\div \tau, e_v^h(0)) + (\dot e_r^h (0), \tau) &= - (A (\dot e_\sigma^P(0)+ 2 \mu \dot e_r^P (0)), \tau), & & \tau \in M_h, \\
(\rho \dot e_v^h (0), w) - (\div e_\sigma^h(0), w) &= -(\rho \dot e_v^P (0), w), & & w \in V_h.
\end{align*}
Recall that $e_v^h(0) = \div e_\sigma^h(0) = 0$ from the choice of $v_{h0}$ in \eqref{did0} and the property of $\sigma_{h0}$ in \eqref{did2}. Furthermore, $(\dot e_r^h(0), \dot e_\sigma^h(0)) = 0$ because $\dot e_\sigma^h(0) \perp K_h$. Thus, taking $\tau = \dot e_\sigma^h(0)$, $w = \dot e_v^h(0)$, and adding the above equations, we have  
\begin{align*}
\| \dot e_\sigma^h (0) \|_A^2 + \| \dot e_v^h (0) \|_{\rho}^2 &= -(A (\dot e_\sigma^P (0) + 2 \mu \dot e_r^P(0)), \dot e_\sigma^h (0)) - (\rho \dot e_v^P (0), \dot e_v^h(0)).
\end{align*}
By the Cauchy--Schwarz inequality and Theorem \ref{ed-proj-thm},
\begin{align*}
(\| \dot e_\sigma^h (0) \|_A^2 + \| \dot e_v^h (0) \|_{\rho}^2)^\half \leq ch^k \| \dot \sigma (0), \dot r (0), \dot v(0) \|_k,
\end{align*}
and \eqref{locking-inter-estm4} follows from $\| \sigma, v, r \|_{W^{1,\infty} H^k} \leq c \| \sigma, v, r \|_{W^{2,1} H^k}$.
\end{proof}

\section{Improved error analysis for the Stenberg and GG elements} \label{section:imp-err-analysis}
The AFW elements have the simplest shape functions of those shown in Table \ref{mixed-approx}, in that they use the space $\mathcal{P}_k$ for stress shape functions, without any additional functions, and for the displacement and rotation shape functions they use $\mathcal{P}_{k-1}$. The Stenberg and GG elements maintain the space $\mathcal{P}_{k-1}$
for the displacement, but uses $\mathcal{P}_k$ for the rotation $r$, and
a space somewhat larger than $\mathcal{P}_k$ for the stress.  For these elements we
can prove one higher order of convergence for $\sigma$ and $r$ than is obtained by the AFW and CGG elements with the same displacement space.  Moreover, a better numerical solution of $u$ can be obtained for these elements via a local post-processing.
\subsection{Improved a priori error estimates}
Since the error analysis for the {\rm GG} and Stenberg elements parallels that for the {\rm AFW} and CGG elements, we avoid repetition and only focus on the steps that require modification.  While the convergence theory for the AFW and CGG element only
required that the density $\rho$ be bounded above and below, in order to obtain the
improved estimates for the Stenberg and GG elements, we require that the density
have bounded derivatives, at least on each element separately (it may
jump across element boundaries).  More precisely, letting $\grad_h$ denote the piecewise gradient operator adapted to the triangulation $\mathcal{T}_h$, we require that
\begin{align} \label{rho-norm}
\| \rho \|_{W_h^{1,\infty}} := \| \rho \|_{L^\infty} + \| \grad_h \rho \|_{L^\infty} <\infty.
\end{align}
Theorem~\ref{ed-semidiscrete-thm-gg} gives main result for the Stenberg and GG elements
from {\rm Table \ref{mixed-approx}}.
\begin{thm} \label{ed-semidiscrete-thm-gg} Let $(M_h, V_h, K_h)$ be the {\rm Stenberg} or {\rm GG} elements of order $k \geq 1$.
Suppose that 
\begin{align} \label{ed-full-conv-cond-gg}
\begin{split}
\sigma, r \in W^{1,1}([0,T_0]; H^{m}), \qquad v \in W^{1,1}([0,T_0]; H^{m-1}), 
\end{split}
\end{align}
for some integer $m$ with $1 \leq m \leq k+1$,
that \eqref{rho-norm} holds, and that the initial data is chosen by \rm{(\ref{did0}--\ref{did3})}. Then the semidiscrete solution $(\sigma_h, v_h, r_h)$ in {\rm (\ref{ed-semidiscrete-eq1}--\ref{ed-semidiscrete-eq3})} satisfies 
\begin{equation}
\label{ed-semidiscrete-estm-gg} 
 \| \sigma - \sigma_h, P_h v - v_h , r - r_h \|_{L^\infty L^2} 
 \leq c h^m (\| \sigma, r \|_{W^{1,1}H^m}  + \| \rho \|_{W_h^{1,\infty}} \| v \|_{W^{1,1}H^{m-1}}), 
\end{equation}
where $c$ depends on $A$ and $\rho_0$. 
\end{thm}
Note that, in this theorem, $m$ may be as large as $k+1$, while in Theorem~\ref{ed-semidiscrete-thm}, $m\le k$.
When $m = k+1$, the estimate \eqref{ed-semidiscrete-estm-gg} show that $v_h$ is
\emph{superclose} to $P_h v$, that is, they are nearer each other than either is
to $v$. As we show in the next section, this can be exploited to define a higher order approximation to $u$ via a local post-process.

To prove the theorem, we decompose the errors into the projection errors $(\es^{P}, e_v^{P}, e_r^{P})$ and the approximation errors $(e_\sigma^h, e_v^h, e_r^h)$ as in (\ref{ed-err-decomp1}--\ref{ed-err-decomp3}), and estimate the two contributions
separately.
\begin{thm} \label{ed-proj-thm-gg} Under the hypotheses of Theorem~\ref{ed-semidiscrete-thm-gg} the following estimates hold. 
\begin{align}
 \label{ed-proj-err1-gg} \| \es^P \|_{L^\infty L^2} &\leq ch^m \| \sigma \|_{L^\infty H^m} , & & 1 \leq m \leq k+1, \\
\label{ed-proj-err2-gg} \| e_v^P \|_{L^\infty L^2} &\leq c h^m \| v \|_{L^\infty H^{m}}  , & & 0 \leq m \leq k, \\
\label{ed-proj-err3-gg} \| e_r^P \|_{L^\infty L^2} &\leq c h^m \| r \|_{L^\infty H^m}  , & & 0 \leq m \leq k+1.
\end{align}
Furthermore, similar inequalities hold with $\sigma$, $v$, $r$ replaced by their
time derivatives.
\end{thm}
The proof is similar to that of Theorem \ref{ed-proj-thm}, and so will be omitted. Note that a better approximation \eqref{ed-proj-err3-gg} in $K_h$ is obtained because the shape functions of $K_h$ for the Stenberg and GG elements of order $k$ are one degree higher than the ones for the AFW and CGG elements of order $k$. 

Now we prove a priori estimates of the approximation errors. 
\begin{thm} Under the hypotheses of Theorem~\ref{ed-semidiscrete-thm-gg}
\begin{equation} \label{ed-theta-estm-gg}
\| e_{\sigma}^h, e_v^h, e_r^h \|_{L^\infty L^2} \leq c h^{m} (\| \sigma, r \|_{W^{1,1}H^m}  + \| \rho \|_{W_h^{1,\infty}} \| v \|_{W^{1,1}H^{m-1}}),
\end{equation}
for $1 \leq m \leq k+1$ where $c$ depends on $A$, $\rho_0$.
\end{thm}
\begin{proof}
Arguing as in the proof of Theorem \ref{ed-approx-thm} we obtain \eqref{ed-evol-eq0}.  Let $\bar\rho$ be the $L^2$-orthogonal projection of $\rho$ into the space of piecewise constant functions associated to the triangulation $\mathcal{T}_h$. Then $\bar\rho \dot e_v^P$
is $L^2$-orthogonal to $V_h$, and therefore $(\rho \dot e_v^P, e_v^h) = ((\rho - \bar\rho) \dot e_v^P, e_v^h)$, so we may obtain from \eqref{ed-evol-eq0} that 
\begin{align*}
\half \frac{d}{dt} (\| \es^h \|_A^2 + \| e_v^h \|_{\rho}^2) \leq  c \| \dot e_\sigma^P, \dot e_r^P, (\rho - \bar\rho) \dot e_v^P \| \,(\| \es^h \|_A^2 + \| e_v^h \|_{\rho}^2)^\half,
\end{align*}
which leads to 
\begin{align*}
(\| e_{\sigma}^h (t) \|_A^2 + \| e_v^h (t) \|_{\rho}^2)^{\half} \leq (\| e_{\sigma}^h(0) \|_A^2 + \| e_v^h (0) \|_{\rho}^2)^{\half} + c\int_0^t \| \dot e_{\sigma}^P, \dot e_r^P, (\rho - \bar\rho) \dot e_v^P \| ds.
\end{align*}
The integral $\int_0^T\|\dot e_\sigma^P,\dot e_r^P\|ds$
and the term $\|e_\sigma^h(0)\|_A$ may be bounded as before,
and again, $e_v^h(0)=0$ because of our choice of initial data.
By the H\"{o}lder inequality, we have  
\begin{align*}
\| (\rho - \bar\rho) \dot e_v^P \| \leq \| \rho - \bar\rho \|_{L^\infty} \| \dot e_v^P \| \leq ch \| \rho \|_{W_h^{1,\infty}} \| \dot e_v^P \|.
\end{align*}
Combining these estimates, we obtain the bound on $e_{\sigma}^h$ and $e_v^h$ in \eqref{ed-theta-estm-gg}.  The bound on $e_r^h$ then follows just as in Theorem~\ref{ed-approx-thm}.    
\end{proof}
Theorem \ref{ed-semidiscrete-thm-gg} follows from Theorems~\ref{ed-theta-estm-gg}
and \ref{ed-proj-thm-gg}. Note that the bound on
$P_h v - v_h = e_v^h$ comes directly from \eqref{ed-theta-estm-gg}.

\subsection{Post-processing}
Let $V_h^*$ be the space of (possibly discontinuous) piecewise polynomials adapted to $\mathcal{T}_h$ of degree $k$ (one degree higher than for $V_h$), and denote
by $\tilde{V}_h$  the orthogonal complement of $V_h$ in $V_h^*$. Denote by $P_h^*$ and $\tilde{P}_h$ the $L^2$-orthogonal projections onto $V_h^*$ and $\tilde{V}_h$, respectively. With $(\sigma_h, v_h, r_h)$ the semidiscrete solution and $u_h$ defined by 
\begin{align} \label{ed-numeric-u}
u_h(t) = u_{h0} + \int_0^t v_h (s)\, ds,
\end{align}
we define $u_h^* \in V_h^*$ at each time $t \in [0,T_0]$ by 
\begin{align}
\label{ed-post-eq1} (\grad_h u_h^*, \grad_h w) &= (A \sigma_h + r_h, \grad_h w), & & w \in \tilde{V}_h, \\
\label{ed-post-eq2} (u_h^*, w) &= (u_h, w), & & w \in V_h.
\end{align}
Note that $V_h^*$ is a discontinuous piecewise polynomial space, so $u_h^*$ can be computed element-wise at relatively little computational cost.
\begin{thm} \label{ed-post-thm}
Let $(\sigma_h, v_h, r_h)$ be the semidiscrete solution for the Stenberg or GG method
of order $k\ge 1$,
and let $u_h$ be defined by \eqref{ed-numeric-u} with $u_{h0}$ chosen so that
$\| P_h u_0 - u_{h0} \| \leq ch^{k+1}$ {\rm (}e.g., $u_{h0}=P_h u_0${\rm )}. Let $u_h^*$ be
defined  by {\rm (\ref{ed-post-eq1}--\ref{ed-post-eq2})}.  Then
\begin{align}
\label{ed-post-error} \| u - u_h^* \|_{L^\infty L^2} \leq ch^{k+1} \| \sigma, v, r \|_{W^{1,1}H^{k+1}}, 
\end{align}
holds with $c$ depending on $A$, $\rho_0$, $\| \rho \|_{W_h^{1,\infty}}$. 
\end{thm}
The proof of this theorem is similar to the post-processing of stationary elasticity problem in \cite{GG10}. A detailed proof can be found in \cite{LeeThesis}. 

\section{Numerical results} \label{section:numerical}
In this section, we present some numerical results supporting the analysis above. As the domain we take the unit square $(0,1) \times (0,1)$ and as finite elements we use the {\rm AFW} element with $k=2$ in the first two examples, and with $k=3$ in the third. In all three examples, we take the material to be homogeneous and isotropic with constant density and, for simplicity, we simply set $\mu=\lambda=\rho=1$.
In each example, we use a temporal discretization method with the same order as the spatial
discretization and with $\Delta t = h$.  All results were implemented using the FEniCS project software \cite{fenicsbook}.

\begin{eg}
In the first example, we take a smooth displacement field which satisfies the homogeneous
displacement boundary conditions:
\begin{align} \label{ed-disp1}
u(t,x,y) =
\begin{pmatrix}
 \sin (\pi x) \sin (\pi y) \sin t \\
 x(1-x)y(1-y) \sin t
\end{pmatrix},
\end{align}
and define $f$ accordingly.  Table \ref{result1} displays the error at time $t=1$, for a
sequence of meshes,
and the observed rates of convergence.  For the numerical method we take
the AFW elements with $k=2$ for spatial discretization, and the Crank--Nicolson scheme 
with $\Delta t=h$ for time discretization which is also second order.  As predicted by Theorem~\ref{ed-semidiscrete-thm} the $L^2$ errors for all variables converge to zero with second order.
\end{eg}
\begin{table}[ht]  
\caption[Numerical result 1 of elastodynamics]{Errors and observed convergence rates for the test problem with exact solution given in \eqref{ed-disp1}.} \label{result1}
\begin{tabular}{>{\small}c >{\small}c >{\small}c >{\small}c >{\small}c >{\small}c >{\small}c >{\small}c >{\small}c }
\hline
\multirow{2}{*}{$\frac 1h$} & \multicolumn{2}{>{\small}c}{$\| \sigma - \sigma_h \|$} & \multicolumn{2}{>{\small}c}{$\| v - v_h\|$} & \multicolumn{2}{>{\small}c}{$\| u - u_h \|$} & \multicolumn{2}{>{\small}c}{$\| r - r_h \| $} \\ 
 & error & order & error & order & error & order & error & order \\ \hline
4 & 5.73e-02 &  --  & 1.03e-02 &  --  & 1.61e-02 &  --  & 2.42e-02 &  --  \\
8 & 1.19e-02 & 1.99 & 2.62e-03 & 1.98 & 4.06e-03 & 1.99 & 6.09e-03 & 1.99 \\
16& 2.78e-03 & 2.00 & 6.57e-04 & 2.00 & 1.02e-03 & 2.00 & 1.52e-03 & 2.00 \\
32& 6.77e-04 & 2.00 & 1.64e-04 & 2.00 & 2.54e-04 & 2.00 & 3.80e-04 & 2.00 \\
64& 1.67e-04 & 2.00 & 4.10e-05 & 2.00 & 6.35e-05 & 2.00 & 9.51e-05 & 2.00 \\ \hline
\end{tabular}                           
\end{table}

\begin{eg}
In this example, the displacement
boundary conditions are inhomogeneous, and so we use the formulation \eqref{ed-inhomog-dirichlet}.
We take an exact solution with limited regularity,
\begin{align} \label{ed-disp2}
u(t,x,y) =
\begin{pmatrix}
(1+t^2) x^{\alpha} y^2 \\
(1+\cos t) x^2 y^{\alpha}
\end{pmatrix} ,
\end{align}
and again define the load accordingly. The fields $v$ and $\sigma$ then belong to $H^{\alpha+1/2-\delta}$ and $H^{\alpha-1/2-\delta}$, respectively, for arbitrary $\delta >0$. Numerical results for several different values of $\alpha$ are shown in Table~\ref{result2}. We see that the convergence rates are somewhat decreased due to the decreased regularity of the solution (but perhaps not as much as might be expected).
\end{eg}

\begin{table}[ht]  
\caption[Numerical result 2 of elastodynamics]{Order of convergence for the exact solution with displacement as in \eqref{ed-disp2} ($\lambda = 1$, $\mu = 1$, $h = \lap t$ and $T_0 = 1$).} \label{result2}
\begin{tabular}{>{\small}c >{\small}c >{\small}c >{\small}c >{\small}c >{\small}c >{\small}c >{\small}c >{\small}c >{\small}c}
\hline
\multirow{2}{*}{$\alpha$} & \multirow{2}{*}{$\frac 1h$} & \multicolumn{2}{>{\small}c}{$\| \sigma - \sigma_h \|$} & \multicolumn{2}{>{\small}c}{$\| v - v_h\|$} & \multicolumn{2}{>{\small}c}{$\| u - u_h \|$} & \multicolumn{2}{>{\small}c}{$\| r - r_h \| $} \\ 
& & error & order & error & order & error & order & error & order \\ \hline
\multirow{5}{*}{$2.2$} 
& 4 & 4.92e-02 &  --  & 4.04e-02 &  --  & 4.22e-02 &  --  & 1.28e-02 &  -- \\
& 8 & 2.23e-02 & 1.14 & 1.07e-02 & 1.91 & 1.04e-02 & 2.02 & 4.28e-03 & 1.58 \\
& 16& 7.37e-03 & 1.60 & 3.67e-03 & 1.55 & 2.52e-03 & 2.05 & 1.45e-03 & 1.56 \\
& 32& 2.37e-03 & 1.63 & 1.26e-03 & 1.54 & 6.31e-04 & 2.00 & 4.18e-04 & 1.79 \\
& 64& 7.60e-04 & 1.64 & 4.13e-04 & 1.61 & 1.58e-04 & 2.00 & 1.22e-04 & 1.78 \\ \hline
\multirow{5}{*}{$2.7$} 
& 4 & 6.92e-02 &  --  & 6.72e-02 &  --  & 4.92e-02 &  --  & 1.61e-02 &  -- \\
& 8 & 2.85e-02 & 1.28 & 9.20e-03 & 2.87 & 1.19e-02 & 2.05 & 4.63e-03 & 1.80 \\
& 16& 7.50e-03 & 1.92 & 1.76e-03 & 2.39 & 2.94e-03 & 2.01 & 1.21e-03 & 1.93 \\
& 32& 1.90e-03 & 1.98 & 4.12e-04 & 2.09 & 7.34e-04 & 2.00 & 3.02e-04 & 2.01 \\
& 64& 4.82e-04 & 1.98 & 1.08e-04 & 1.93 & 1.83e-04 & 2.00 & 8.03e-05 & 1.91 \\ \hline
\multirow{5}{*}{$3.2$} 
& 4 & 1.14e-01 &  --  & 1.05e-01 &  --  & 5.70e-02 &  --  & 2.55e-02 &  -- \\
& 8 & 4.41e-02 & 1.36 & 1.49e-02 & 2.81 & 1.37e-02 & 2.05 & 7.65e-03 & 1.73 \\
& 16& 1.17e-02 & 1.91 & 2.68e-03 & 2.47 & 3.42e-03 & 2.01 & 2.09e-03 & 1.87 \\
& 32& 2.96e-03 & 1.99 & 6.07e-04 & 2.15 & 8.51e-04 & 2.00 & 5.21e-04 & 2.01 \\
& 64& 7.39e-04 & 2.00 & 1.57e-04 & 1.95 & 2.13e-04 & 2.00 & 1.33e-04 & 1.97 \\ \hline
\end{tabular}
\end{table} 

\begin{eg} \label{ed-eg-radau}
In the final example, we consider a third order method.  For spatial discretization
we use the AFW method with $k=3$, and for time discretization we use the 2-stage RadauIIA method which is a third-order implicit Runge--Kutta methods with the Butcher tableau shown in Table~\ref{RadauIIA-butcher}.

In the previous examples, $u_h$ is obtained by a simple numerical time integration of $v_h$ based on the trapezoidal rule. However, the trapezoidal rule gives only second order convergence in $\lap t$, which is lower than the convergence rates of other unknowns. To achieve third order convergence of $\| u - u_h \|$ a numerical integration of $v_h$, exploiting additional numerical data generated by the RadauIIA method, is needed. In Table \ref{RadauIIA-butcher}, the RadauIIA method at $i$th time step ($t = i \lap t$) generates an auxiliary numerical data approximating $\dot v ((i + 1/3) \lap t )$, which will be denoted by $V_t^{i+1/3}$. Let $V^i$ be the $i$th numerical velocity obtained by the RadauIIA method and $U^0$ be the numerical initial displacement such that $\| u(0) - U^0 \| \leq ch^3$. Then, regarding Taylor expansion
\begin{align*}
g(a) = g(0) + a g'(0) + \frac {a^2}2 g''(a/3) + o(a^3), 
\end{align*}
the numerical integration for reconstruction of $u_h$ is inductively defined by 
\begin{align*}
 U^{i+1} = U^i + \lap t  V^i + \frac {\lap t^2}2 V_t^{i+\frac 13}, \qquad i \geq 0.
\end{align*}
The numerical results in Table \ref{RadauIIA} show that the expected third order convergence rates are obtained for all errors.

\begin{table}[h] 
\caption[The Butcher tableau for the 2-stage RadauIIA scheme]{The Butcher tableau for the 2-stage RadauIIA Runge--Kutta scheme.} \label{RadauIIA-butcher}
\begin{center}
\begin{tabular} {c| c c}
$1/3$ &$5/{12}$ &$- 1/{12}$ \\ 
$1$ & $3/4$ & $1/4$ \\
\hline 
 & $3/4$ &$1/4$
\end{tabular}
\end{center}
\end{table}

\begin{table}[!t]  
\caption[Numerical result 4 of elastodynamics]{Order of convergence for the exact solution with displacement in \eqref{ed-disp1} ($\lambda = 1$, $\mu = 1$, $h = \lap t$ and $T_0 = 1$). The AFW elements with $k=3$ and the 2-stage RadauIIA time discretization are used.} \label{RadauIIA}
\begin{tabular}{>{\small}c >{\small}c >{\small}c >{\small}c >{\small}c >{\small}c >{\small}c >{\small}c >{\small}c}
\hline
\multirow{2}{*}{$\frac 1h$} & \multicolumn{2}{>{\small}c}{$\| \sigma - \sigma_h \|$} & \multicolumn{2}{>{\small}c}{$\| v - v_h\|$} & \multicolumn{2}{>{\small}c}{$\| u - u_h \|$} & \multicolumn{2}{>{\small}c}{$\| r - r_h \| $} \\ 
 & error & order & error & order & error & order & error & order \\ \hline
4 & 1.31e-02 &  --  & 1.38e-03 &  --  & 2.10e-03 &  --  & 3.77e-03 &  --  \\
8 & 1.02e-03 & 3.68 & 1.78e-04 & 2.96 & 2.58e-04 & 3.02 & 4.18e-04 & 3.17 \\
16& 9.88e-05 & 3.37 & 2.23e-05 & 3.00 & 3.25e-05 & 2.99 & 5.05e-05 & 3.05 \\
32& 1.14e-05 & 3.12 & 2.78e-06 & 3.00 & 4.09e-06 & 2.99 & 6.28e-06 & 3.01 \\
64& 1.40e-06 & 3.03 & 3.46e-07 & 3.00 & 5.13e-07 & 2.99 & 7.85e-07 & 3.00 \\ \hline
\end{tabular}
\end{table}
 
\end{eg}

%

\bibliography{FEM}
\bibliographystyle{amsplain}
\addcontentsline{toc}{chapter}{Bibliography}

\vspace{.125in}

\end{document}